\documentclass[letterpaper, 10 pt, conference]{ieeeconf}  % Comment this line out if you need a4paper

\IEEEoverridecommandlockouts                              % This command is only needed if
\usepackage{amsfonts}
\usepackage{amsmath, scalefnt}
\usepackage{rotating}

\usepackage{multirow}
\usepackage{graphicx}
\usepackage{color}
\usepackage{epstopdf}
\newcommand{\R}{\mathbb{R}}
\renewcommand{\S}{\mathbb{S}}
\newcommand{\N}{\mathbb{N}}
\newcommand{\n}{\mathbf{n}}
\newcommand{\m}{\mathbf{m}}
\newcommand{\e}{\mathbf{e}}
\newcommand{\f}{\mathbf{f}}
\newcommand{\bmat}[1]{\begin{bmatrix}#1\end{bmatrix}}
\newcommand{\norm}[1]{\left\lVert{#1}\right\rVert}

\newcommand{\ip}[2]{\left\langle #1, #2 \right\rangle}

\newcommand{\mcl}[1]{\mathcal{ #1}}
\newcommand{\mbf}[1]{\mathbf{ #1}}

\newcommand{\D}{\Delta}

\newtheorem{thm}{Theorem}
\newtheorem{defn}[thm]{Definition}
\newtheorem{lem}[thm]{Lemma}

\newtheorem{cor}[thm]{Corollary}

\let\bbbl\biggl
\let\bbbbl\Biggl

\let\bbbr\biggr
\let\bbbbr\Biggr

\title{\LARGE \bf
Integral Quadratic Constraints with Infinite-Dimensional Channels
}

\author{Aleksandr Talitckii$^{1}$ and Peter Seiler$^{2}$ and Matthew M. Peet$^{1}$  % <-this % stops a space
\thanks{$^{1}$Aleksandr Talitckii and Matthew M. Peet are with the School for the Engineering of Matter, Transport and Energy, Arizona State University, Tempe, AZ, 85298, USA. This works is supported by the National Science Foundation under grants NSF CMMI-1935453 and NSF CMMI-1931270 
        {\tt\small atalitck@asu.edu} and {\tt\small mpeet@asu.edu }}%
\thanks{$^{2}$Peter Seiler is with the Electrical Engineering and Computer Science, University of Michigan, Ann Arbor, MI 48109, USA
        {\tt\small pseiler@umich.edu}}%
}

\begin{document}

\maketitle
\thispagestyle{empty}
\pagestyle{empty}

%%%%%%%%%%%%%%%%%%%%%%%%%%%%%%%%%%%%%%%%%%%%%%%%%%%%%%%%%%%%%%%%%%%%%%%%%%%%%%%%
\begin{abstract}

Modern control theory provides us with a spectrum of methods for studying the interconnection of dynamic systems using input-output properties of the interconnected subsystems. Perhaps the most advanced framework for such input-output analysis is the use of Integral Quadratic Constraints (IQCs), which considers the interconnection of a nominal linear system with an unmodelled nonlinear or uncertain subsystem with known input-output properties. Although these methods are widely used for Ordinary Differential Equations (ODEs), there have been fewer attempts to extend IQCs to infinite-dimensional systems. In this paper, we present an IQC-based framework for Partial Differential Equations (PDEs) and Delay Differential Equations (DDEs). First, we introduce infinite-dimensional signal spaces, operators, and feedback interconnections. Next, in the main result, we propose a formulation of hard IQC-based input-output stability conditions, allowing for infinite-dimensional multipliers. We then show how to test hard IQC conditions with infinite-dimensional multipliers on a nominal linear PDE or DDE system via the Partial Integral Equation (PIE) state-space representation using a sufficient version of the Kalman-Yakubovich-Popov lemma (KYP). The results are then illustrated using four example problems with uncertainty and nonlinearity.

% We present an extension of the hard Integral Quadratic Constraints (IQCs) approach to infinite-dimensional systems such as Partial Differential Equations (PDEs) and Delay Differential Equations (DDEs). The first result is a formalism of the IQC framework and the stability theorem, where we provide some definitions for general signal spaces and present an hard formulation of IQCs in the infinite-dimensional case. This leads to the stability theorem, which is a new way of verifying robustness analysis for PDEs. The second results show a numerical way to check the theorem conditions. Partial Integral Equalities (PIEs) are used for numerical analysis of the system. We formulate and prove a sufficient version of Kalman-Yakubovic-Popov (KYP) lemma that allows us to use MATLAB toolbox PIETOOLS. Finally, we apply our methods to several examples of ODE and DDE and present some IQCs for some standard type uncertainties.

\end{abstract}

%%%%%%%%%%%%%%%%%%%%%%%%%%%%%%%%%%%%%%%%%%%%%%%%%%%%%%%%%%%%%%%%%%%%%%%%%%%%%%%%
\section{Introduction}
As developed in the 1970's and best exemplified by transfer-function-based properties of passivity and the small-gain condition,  the Input-Output framework was a response to the increasing complexity of circuit-based subsystems. This framework obviated the need for a precise system model by characterizing the input-output behaviour of a system in terms of the input-output behaviour of its subsystems. However, by completely eliminating the model, and by only considering a subset of input-output properties, passivity and small-gain conditions resulted in substantially conservative results.

 An attempt to improve the accuracy of the input-output framework was the use of multipliers proposed by Zames-Falb~\cite{zames1968stability}, Yakubovich~\cite{yakubovich1967frequency} and others. However, verification of these multiplier-based conditions proved difficult. The Integral Quadratic Constraints (IQC) framework, introduced by Megretski and Rantzer~\cite{megretski1997system}, provided an attempt to simplify the multiplier-based input-output framework while also integrating modern model-based computational methods such as Linear Matrix Inequalities (LMIs) via generalizations of the Kalman-Yakubovich-Popov (KYP) lemma~\cite{rantzer1996kalman, gusev2006kalman}.  While this framework originally required homotopy in the unmodeled subsystem, recent works~\cite{carrasco2019conditions, seiler2014stability} have attempted to remove the homotopy condition -- thereby allowing for analysis of known nonlinear subsystems.

	Despite the success of the IQC framework, its application to delayed and PDE systems has been limited. Specifically, most work on this topic treats the delayed or PDE dynamics as an unknown subsystem, with certain characterization of its input-output behaviour (e.g. ~\cite{pfifer2015integral,peet2009positive,pfifer2015robustness}). While this was a reasonable approach at a time when analysis of linear delayed and PDE systems~\cite{fridman2020delayed} was considered computationally challenging, recent work has shown that model-based computational evaluation of the input-output properties of linear delayed and PDE systems can be performed efficiently and accurately~\cite{briat2014linear, bondarko2003necessary}. As a result, the framework for IQC-based analysis has shifted, where now delayed and PDE components are contained in the nominal subsystem and nonlinearities and uncertainties are isolated in the unmodelled subsystem. This paradigm shift, however, means that the interconnecting signals between nominal and unknown subsystem may now be infinite-dimensional.

Recent attempts to consider delayed and PDE models in the known subsystem include the projection-based approach in~\cite{barreau2020integral} (wherein the interconnection signals are finite-dimensional) and the Sum-of-Squares-based dissipation inequalities in~\cite{ahmadi2016dissipation} (wherein the interconnection signals are infinite-dimensional). Neither of these results, however, directly consider the problem of extension of the IQC framework to subsystems interconnected by infinite-dimensional signal spaces.

The goal of this paper, therefore, is to propose a generalization of the hard IQC framework for a nominal infinite dimensional system interconnected with a nonlinear or uncertain subsystem by infinite-dimensional signals. We will accomplish this goal in three steps. First, we extend the IQC framework to infinite-dimensional systems, signals, interconnections, and multipliers, and generalize an IQC stability theorem to such interconnections.

Next, we assume both the nominal subsystem and multiplier can be represented as a Partial Integral Equation (PIE) as discussed in section~\ref{sec:PIETOOLS}. PIE representations exist for most infinite-dimensional linear systems, including those with delays and those governed by PDEs. The existence of a PIE representation allows hard IQC conditions to be tested numerically using algorithms for the optimization of positive Partial Integral operators. Based on this PIE representation, we extend the KYP Lemma and use this extension to propose convex tests for conditions of the IQC theorem to be satisfied.

Finally, we examine several classes of nonlinearity and uncertainty with infinite-dimensional inputs and outputs and show that they satisfy a generalized version of the hard IQC constraints typically used for finite-dimensional systems.

Having completed these three steps, we then apply the results to several specific examples of delayed and PDE systems and show that the proposed approach is an improvement over alternatives such as quadratic stability.

\section{Notation and Signal Spaces}
\label{sec:notations}
\paragraph{Notation} We denote by $\N,\R, \S^n, I$ and $\mbf 0$ the natural numbers, the real numbers, the space of $n\times n$ symmetric matrices, the identity operator, and the null operator, respectively. For $\Omega \subset \R$, $L^n_2(\Omega)$ denotes the space of Lebesgue square integrable functions $f: \Omega \rightarrow \R^n$ with inner-product $\ip{f}{g}_{L_2} = \int_\Omega f(s)^T g(s) ds$.  For Hilbert space $\mbf H$ we use $\mbf L_{\mbf H}$  to denote the extension of $L_2$ to square-integrable functions
 $\mbf u: [0,\infty) \rightarrow \mbf H$. $\mbf L_{\mbf H}$ is itself a Hilbert space~\cite{hille1948functional} with associated inner product
\[
\left<\mbf u, \mbf v\right>_{\mbf L_{\mbf H}} = \int_0^\infty  \ip{\mbf u(t)}{ \mbf v(t)}_{\mbf H} dt.
\]
Clearly, if $\mbf H = \R^n$, then $\mbf L_{\mbf H}=L^n_2[0, \infty)$. Associated with $\mbf L_{\mbf H}$, we define an extended space, $\mbf L_{e, \mbf H}$, of functions,  ${\mbf u}:[0, \infty) \rightarrow \mbf H$ such that for any $T\geq 0$, we have that
\[
\int_0^T  \norm{\mbf u(t)}^2_{\mbf H} dt =
\int_0^T  \ip{\mbf u(t)}{\mbf u(t)}_{\mbf H} dt
\]
is finite.
%
%\begin{defn}[Extended Spaces]
%We say that a function ${\mbf u}:[0, \infty) \rightarrow \mbf H$ belongs to the extended space $\mbf L_{e, \mbf H}$ if for all $T\geq 0$
%\[
%\int_0^T  \norm{\mbf u(t)}^2_{\mbf H} < \infty
%\]
%\end{defn}
%

Given $a,b \in \R$ and $\mbf  n \in \N^2$, we denote the Hilbert space $\mbf Z^{\mbf  n} := \R^{n_1} \times L^{n_2}_{2}[a, b] $ (with inner product $\ip{(u, U)}{(v, V)}=u^Tv+\ip{U}{V}_{L_2}$) and the extended signal space ${\mbf L}^{\mbf n}_{e, [a, b]}:={\mbf L}_{e, \mbf Z^{\mbf n}}$.

For notational convenience, given $\mbf u(t) = (u(t),  U(t)) \in \mbf Z^{\n} $ and $\mbf v(t) = (v(t),  V(t)) \in \mbf Z^{\m} $, we define the component-wise concatenation of $\mbf u(t), \mbf v(t)$ as
\[
\bmat{\mbf u(t) \\ \mbf v(t)} :=\left (\bmat{u(t) \\ v(t)}, \bmat{U(t) \\  V(t)}\right) \in \mbf Z^{\n + \m}.
\]
Moreover, given $\mbf u \in \mbf L^{\n}_{e, [a, b]}$ and $\mbf v \in \mbf L^{\m}_{e, [a, b]}$, we use the notation
\[
\bmat{\mbf u \\ \mbf v}(t) := \bmat{\mbf u(t) \\ \mbf v(t)} \in \mbf Z^{\n + \m}
\]
\paragraph{Operators} For any Hilbert space, $\mbf H$, we define the truncation operator $P_\tau:\mbf L_{e, \mbf H} \rightarrow \mbf L_{e, \mbf H}$ for any $\mbf y \in \mbf L_{e, \mbf H}$ as
\[
(P_{\tau}\mbf y)(t) = \begin{cases}\mbf y(t),\;\; 0\leq t\leq \tau \\ 0, \;\;\; \text{otherwise.}\end{cases}\\
\]

%The key fundamental properties of operators are presented below

\begin{defn}
Let $\mbf H, \mbf G$ be Hilbert spaces, then an operator $G:\mbf L_{e,\mbf H}  \rightarrow \mbf L_{e,\mbf G} $ is
\begin{enumerate}
\item \textbf{Causal} if $P_\tau G = P_\tau G P_\tau.$ for any $\tau\ge 0$.
\item \textbf{Bounded  on $\mbf L_{\mbf H}$} if there exist $C\ge 0$ such that for all ${\mbf v}\in \mbf L_{\mbf H}  $, we have that $\|G{\mbf v}\|_{\mbf L_{\mbf G}} \leq C \| {\mbf v}\|_{\mbf L_{\mbf H}}$.
\item \textbf{Bounded  on $\mbf L_{e,\mbf H}$} if there exist $C\ge 0$ such that for all ${\mbf v}\in \mbf L_{e, \mbf H}  $, we have  that $\|P_\tau G{\mbf v}\|_{\mbf L_{\mbf G}} \leq C \| P_\tau {\mbf v}\|_{\mbf L_{\mbf H}} $ for all $ \tau \ge 0$.
\item \textbf{Incrementally  $\mbf L_{e,\mbf H}$-bounded} if there exist $C\ge 0$ such that for all ${\mbf v}, {\mbf u}\in \mbf L_{e, \mbf H} $, we have that
$\|P_\tau (G{\mbf v} - G{\mbf u})\|_{\mbf L_{\mbf G}} \leq C \| P_\tau ({\mbf v} - {\mbf u})\|_{\mbf L_{\mbf H}}  $  for all $ \tau \ge 0$.
\end{enumerate}

\end{defn}
For causal operators, bounded on $\mbf L_{\mbf H}$ is equivalent to bounded  on $\mbf L_{e,\mbf H}$. For a causal linear operator, bounded on $\mbf L_{\mbf H}$ is equivalent to incrementally $\mbf L_{e,\mbf H}$-bounded.

%
%
%In this paper, we assume that all our operators are causal for simplicity.
%
%The second property is Boundness on $L_{\mbf H}$ or $L_{e, \mbf H}$. But since the operator is causal and
%\[
%\|P_\tau \Delta v \|= \|P_\tau \Delta P_\tau v\| \leq \|\Delta P_\tau v\|  \leq C \|P_\tau v\|,
%\]
%they are equivalent.
The set of all causal, bounded linear operators between Hilbert spaces, $\mbf H_1$ and $\mbf H_2$, is denoted $\mcl L(\mbf H_1,\mbf H_2)$ and is a Banach space with induced norm~\cite{hille1948functional}. We denote $\mcl L(\mbf H):=\mcl L(\mbf H,\mbf H)$.

Given a bounded linear operator $\mcl K:\mbf H \rightarrow \mbf H$, we define the associated multiplication operator $M_{\mcl K}:\mbf L_{e,\mbf H} \rightarrow \mbf L_{e,\mbf H}$ for $w \in \mbf L_{e,\mbf H}$ by
\[
(  M_{\mcl K} \mbf w)(t):=\mcl K \mbf w(t)\in \mbf H.
\]
\paragraph{PI operators}
We say $\mcl P$ is a Partial Integral (PI) operator on $\mbf Z^n$,  if there exists matrix $P$, bounded functions $Q_1,Q_2, R_0$, and separable functions $R_1,R_2$ such that
\begin{equation*}
\left(\mcl P\bmat{x \\ \mbf x}\right)(s) := \bmat{Px + \int_a^b Q_1(s) \mbf x(s)ds \\ Q_2(s) x + (\mcl P_{\{R_i\}} \mbf X)(s)},
\end{equation*}
where
\begin{align*}
 \mcl P_{\{R_i\}} \mbf x)(s) & :=   \\
    & \hspace{-15mm} R_0(s) \mbf x(s) + \int_a^s R_1(s, \theta) \mbf x(\theta) d\theta + \int_s^b R_2(s, \theta) \mbf x(\theta) d\theta.
\end{align*}
We denote the set of PI operators by $\Pi_4$
%{\color{red} The set of PI operators is denoted by $\Pi_4$}

Given matrix $P$, bounded functions $Q_1,Q_2, R_0$, and separable functions $R_1,R_2$, the associated PI operator is denoted $\left(\mcl P \bmat{P& Q_1 \\ Q_2 & \{R_i\}} \right)\in \Pi_4$. The set of PI operators form a $*$-algebra of bounded linear operators as discussed in, e.g.~\cite{peet2021representation}.

\section{Feedback Interconnections on Hilbert Space}

\label{sec:problem}
In this section, we consider basic definitions of the interconnection of systems $G, \D$ for the case when $\mbf H=\mbf Z$ where recall that $\mbf Z^{\n}  :=\R^{n_1} \times \mbf L_2^{n_2}[a, b] $.

\begin{defn}[Interconnection of $G$ and $\Delta$]
Given operators $G:\mbf L^{\mbf n}_{e, [a, b]} \rightarrow \mbf L^{\mbf m}_{e, [a, b]}$ and  $\Delta:\mbf L^{\mbf m}_{e, [a, b]} \rightarrow \mbf L^{\mbf n}_{e, [a, b]}$,
and signals $\e\in \mbf L^{\mbf n}_{e, [a, b]}$ and $\f\in \mbf L^{\mbf m}_{e, [a, b]}$, we say that $\mbf u \in \mbf L^{\mbf n}_{e, [a, b]},   \mbf v \in \mbf L^{\mbf m}_{e, [a, b]}$ \textbf{satisfy the interconnection defined by $[G, \D]$} if
\begin{equation}
{\mbf v} = G {\mbf u} +\f \qquad  \text{and} \qquad {\mbf u} = \Delta {\mbf v} + \e. \label{eqn:interconnection}
\end{equation}
% The set of solutions of operator $G$ are $\{({\mbf v}, {\mbf u}) | {\mbf v} = G{\mbf u},\; {\mbf u}\in L_{2e} \}$
\end{defn}

Typically, $G$ is a known causal bounded linear operator and $\D$ is either nonlinear or unknown, but lies in some set $\Delta \in \mbf \Delta$ with known input-output properties. For a given $G$ and $\D$, we define the following notion of well-posedness of the feedback interconnection, guaranteeing existence and uniqueness of a causal mapping from inputs $\mbf e, \mbf f$ to outputs $\mbf u, \mbf v$.

%In general the interconnection $[G, \D]$ is not well-defined, so we use a well-studied notion of the well-posedness, that we interpret in terms of operators. In this study, we assume that all interconnections are well-posed.

\begin{defn}[Well-posedness]
Given operators $G:\mbf L^{\mbf n}_{e, [a, b]} \rightarrow \mbf L^{\mbf m}_{e, [a, b]}$ and  $\Delta:\mbf L^{\mbf m}_{e, [a, b]} \rightarrow \mbf L^{\mbf n}_{e, [a, b]}$, we say the interconnection defined by $[G,\D]$ is well-posed if for any $\e  \in \mbf L^{\mbf n}_{e,  [a, b ]}, \f \in \mbf L^{\mbf n}_{e,  [a, b ]}$, we have the following.
\begin{enumerate}
\item \textbf{Existence and Uniqueness:} There exist unique  ${\mbf u}\in\mbf L^{\mbf n}_{e,  [a, b ]}, {\mbf v} \in \mbf L^{\mbf m}_{e,  [a, b]}$ such that  ${\mbf u},{\mbf v}$ satisfy the interconnection defined by $[G, \D]$.
\item \textbf{Causality:} If ${\mbf u},{\mbf v}$  satisfy the interconnection defined by $[G, \D]$  and $\hat {\mbf u},\hat {\mbf v}$ satisfy the interconnection defined by $ [ G,  \D]$ for $P_\tau \e, P_\tau \f$ for some $\tau \ge 0$, then $P_\tau({\mbf u}-\hat {\mbf u})=0$ and $P_\tau({\mbf v}-\hat {\mbf v})=0$.
\end{enumerate}
%furthermore
%if  $Sol =\{(({\mbf u}, {\mbf v}), (e, f))\;| (({\mbf u}, {\mbf v}),  (e, f))\in L_{2e} \text{ satisfy } (1)\}$  defines a causal operator $[G, \D] : (f,e) \xrightarrow{} (u, {\mbf v})$ %, such that $\forall (e,f) \in L_{2e}$ there exists unique $(u, {\mbf v})\in L_{2e}$ such that $((u, {\mbf v}), (e, f)) \in Sol$
\end{defn}

\textbf{Notation:} Given $G, \D:\mbf L_{e, [a, b]} \rightarrow \mbf L_{e, [a, b]} $ if the interconnection defined by $[G,\D]$ is well-posed, then for $\e,\f \in \mbf L_{e, [a, b]} $, we say that
\[
\bmat{\mbf u \\ \mbf v}=\mcl F_{G,\D}\left(\bmat{\e\\ \f}\right)
\]
if $\mbf u,\mbf v$ satisfies the interconnection defined by $[G, \D]$.
\begin{defn}
We say \textbf{the feedback system defined by $[G, \Delta]$ is stable} if the \emph{interconnection defined by $[G, \Delta]$ is well-posed} and $\mcl F_{G,\D}$ is bounded on $\mbf L^{\mbf n+\mbf m}_{\mbf Z}$ where recall $\mcl F_{G,\D}$ is bounded on $\mbf L^{\mbf n+\mbf m}_{\mbf Z}$ if there exists a $C$ such that for any $\e \in \mbf L^{\mbf n}_{\mbf Z},  \f \in \mbf L_{\mbf Z}^{\m}$, if \emph{${\mbf u},{\mbf v}$ satisfy the interconnection defined by $[G,\D]$} , then
\[
\norm{ \bmat{{\mbf v}\\ {\mbf u}}}_{\mbf L} < C \norm{\bmat{\f\\ \e}}_{\mbf L}
\]
\end{defn}
\vspace{1mm}
\subsection{Integral Quadratic Constraints}
Next, we extend the Hard IQC framework to infinite-dimensional systems.
\begin{defn} \label{def:IQC}
We say the operator $\Delta: \mbf L^\m_{e, [a, b]}  \rightarrow \mbf L^\n_{e, [a, b]} $ \textbf{ satisfies the hard IQC} defined by operators $\Psi: \mbf L^{\n+\m}_{e, [a, b]} \rightarrow \mbf L^{\n+\m}_{e, [a, b]}$ and $ \mcl K: \mbf Z^{\n + \m}\rightarrow \mbf Z^{\n + \m}$, if for any $T> 0$ and for all $\mbf v \in \mbf L^{\n + \m}_{e, [a ,b]} $
\begin{equation}
\left<P_T\Psi \bmat{ {\mbf v}\\ \D {\mbf v}},  P_T M_{\mcl K} \Psi\bmat{ {\mbf v}\\ \D {\mbf v}}\right>_{\mbf{L}} \hspace{-2mm}\ge 0, \label{IQC:D}
\end{equation}
where for all $\mbf w\in \mbf L^{\n+\m}_{e, [a ,b]} $
\[
( M_{\mcl K} \mbf w)(t):=\mcl K \mbf w(t)\in \mbf Z^{\n + \m}.
\]
\end{defn}
\vspace{3mm}

\subsection{Problem formulation}
Suppose we are given a known linear system/operator $G$ and a set of nonlinear systems/operators $\mbf \Delta$ where the ``graph'' of every $\Delta \in \mbf \Delta$, defined as
\[
\left\{\bmat{\mbf v \\ \Delta \mbf v} \in \mbf L^{\mbf n+\mbf m}_{e,[a,b]}\;:\; \mbf v \in \mbf L^{\mbf n}_{e,[a,b]}\right\},
\]
is known to satisfy a set of certain Integral Quadratic Constraints (IQCs) as parameterized the set of operators ${\mbf K}$ and $\mbf \Psi$  -- See Definition~\ref{def:IQC}. Our goal is to show that if the inverse graph of $G$ satisfies a similar IQC for some $(\mcl K,\Psi) \in {\mbf K} \times \mbf \Psi$, then the feedback interconnection of $G$ and $\Delta$ is stable for all $\Delta \in \mbf \Delta$.

\section{The Main Theorem}
\label{sec:IQC}
As discussed in the previous section, we would like to show that if the graph of $\Delta$ and the inverse graph of $G$ are separated by some quadratic form defined by $\mcl K$ and $\Psi$, then the feedback interconnection of $G$ and $\Delta$ is stable. This result is given by Theorem~\ref{thm:IQC}, the proof of which is a generalization of the technique used in~\cite{khong2021integral} and~\cite{teel1996graphs}.%In this case,% the quadratic form is defined by a hard IQC.
\begin{thm}[IQC theorem]\label{thm:IQC}
Suppose $G:\mbf L_{e, [a,b]}^\n \rightarrow \mbf L_{e, [a,b]}^\m$ is bounded,  $\D:\mbf L_{e, [a,b]}^\m \rightarrow \mbf L_{e, [a,b]}^\n $ and the interconnection defined by $[G, \D]$ is well-posed.

Further suppose
%\begin{itemize}
%    \item The interconnection defined by $[G, \D]$ is well-posed
%    \item
     there exists a causal, incrementally bounded on $\mbf L_{e,[a,b]}^{\n+\m} $ operator $\Psi$, and ${\mcl K} \in \mcl L(\mbf Z^{\n +\m}$) such that
     \begin{enumerate} \item $\D$ satisfies the hard IQC defined by $\Psi, \mcl K$.
     \item For any $ \mbf u \in \mbf L^{\n}_{e, [a, b]} , \mbf v \in \mbf L^{\m}_{e, [a, b]}$ we have that
\begin{equation}
      %& \left<P_T \Psi \bmat{ {\mbf v}\\ \D {\mbf v}}, P_T M_{\mcl K} \Psi \bmat{ {\mbf v}\\ \D {\mbf v}}\right>_{\mbf L} \geq 0 \label{IQC:D}  \\
     \hspace{-10mm} \left<P_T \Psi \bmat{ G{\mbf u}\\ {\mbf u}}, P_T M_{\mcl K} \Psi \bmat{ G{\mbf u}\\ {\mbf u}}\right>_{\mbf L} \leq -\varepsilon\|{P_T \mbf u}\|^2_{\textbf{L}},  \label{IQC:G}
\end{equation}
    % \left<\Psi \bmat{ {\mbf v}\\ \D {\mbf v}}, {\mcl K} \bmat{ {\mbf v}\\ \D {\mbf v}}\right>_T \geq 0 \;\;\text{ for all } {\mbf v}\in \mbf L_{e, [a,b]}^n[0, \infty)
    % \]
    % \[
    % \left<\Psi \bmat{ G{\mbf u}\\ {\mbf u}}, {\mcl K} \bmat{ G{\mbf u}\\ {\mbf u}}\right>_T \leq -\{\mbf v}arepsilon\|{\mbf u}_T\|_{L_{H,e}} \;\;\text{ for all } {\mbf u}\in \mbf L_{e, [a,b]}^n[0, \infty)
    % \]
%\end{itemize}
for all $T>0$.
 \end{enumerate}
  Then the feedback system defined by $[G, \D]$ is stable.
%We recall that ${M_{\mcl K}}:\mbf L^{\n+\m}_{e, [a, b]}\rightarrow \mbf L^{\n+\m}_{e, [a, b]}$ is the multiplication operator defined by multiplier ${\mcl K}$, then the feedback system defined by $[G, \D]$ is stable.

\end{thm}
\begin{proof}
% {\color{red}  $\|M_{{\mcl K}}\|_{L_2}=\|{\mcl K}\|_Z, C_{M_{{\mcl K}}}, C_\Psi$?}
Define the shorthand notation ${\mbf u}_T:=P_T\mbf u$ and $\ip{\mbf u}{\mbf v}_T=\ip{P_T \mbf u}{P_T \mbf v}_{\mbf L}$. Clearly, by Cauchy Schwartz, $\ip{\mbf u}{\mbf v}_T\le \norm{{\mbf u}_T}_{\mbf L}\norm{{\mbf v}_T}_{\mbf L}$.

Now, because the interconnection defined by $[G,\D]$ is well posed, the feedback system defined by $[G,\D]$ is stable if there exists a $C>0$ such that for any $\e,\f \in \mbf L_{\mbf Z}$, if $\mbf u,\mbf v$ satisfy the interconnection defined by $[G,\D]$,
\[\norm{\bmat{\mbf u\\\mbf v}}_{\mbf L_{\mbf Z}}\le C\norm{\bmat{\e\\ \f}}_{\mbf L_{\mbf Z}}.\]

Now, for $\e, \f \in \mbf L_{[a,b]}$, suppose that $\mbf u, \mbf v$ satisfy the interconnection defined by $G,\D$. Then $\mbf v=G\mbf u+\f$ and $\mbf u=\Delta \mbf v+\e$ and from equation~\eqref{IQC:D},

\begin{align*}
&\left<\Psi \bmat{ {\mbf v}\\ \D {\mbf v}}, M_{\mcl K} \Psi  \bmat{ {\mbf v}\\ \D {\mbf v}}\right>_T + \left<\Psi \bmat{ G {\mbf u}\\  {\mbf u}}, M_{\mcl K} \Psi  \bmat{ G {\mbf u}\\  {\mbf u}}\right>_T \\
  &\qquad  - \left<\Psi \bmat{ G {\mbf u}\\  {\mbf u}}, M_{\mcl K} \Psi  \bmat{ G {\mbf u}\\  {\mbf u}}\right>_T\ge 0.
\end{align*}
Subtracting Eqn.~\eqref{IQC:G} from this expression, we find
\begin{align*}
 \varepsilon\|{\mbf u}_T\|_{\mbf L}& \hspace{-1.0mm}\leq \hspace{-0.5mm} \left<\hspace{-1mm}\Psi \hspace{-0.5mm}\bmat{ {\mbf v}\\ \D {\mbf v}}\hspace{-1mm}, \hspace{-0.5mm}M_{\mcl K} \hspace{-0.5mm}\Psi \hspace{-0.5mm} \bmat{ {\mbf v}\\ \D {\mbf v}}\right>_T \hspace{-2.5mm} \hspace{-0.5mm}  - \hspace{-0.5mm} \left<\hspace{-1mm}\Psi\hspace{-0.5mm} \bmat{ G {\mbf u}\\  {\mbf u}}\hspace{-1mm},\hspace{-0.5mm} M_{\mcl K}\hspace{-0.5mm} \Psi\hspace{-0.5mm}  \bmat{ G {\mbf u}\\  {\mbf u}}\right>_T \\
 & \hspace{-8mm}= \left<\Psi \bmat{ {\mbf v} \\ \D {\mbf v}} -\Psi \bmat{ G {\mbf u}\\ {\mbf u}}, M_{\mcl K} \Psi  \bmat{ {\mbf v} \\ \D {\mbf v}} -M_{\mcl K} \Psi \bmat{ G {\mbf u}\\ {\mbf u}}\right>_T \\
 & \hspace{-2mm}+ \left<\Psi \bmat{ G {\mbf u} \\  {\mbf u} }, M_{\mcl K} \Psi  \bmat{{\mbf v} \\ \D {\mbf v}} - M_{\mcl K} \Psi \bmat{ G {\mbf u}\\ {\mbf u}}\right>_T\\
  &+ \left<\Psi \bmat{{\mbf v} \\ \D {\mbf v}} - \Psi \bmat{  G {\mbf u}\\  {\mbf u}}, M_{\mcl K} \Psi  \bmat{ G {\mbf u}\\ {\mbf u}}\right>_T.
\end{align*}
Now,  because $\Psi$ and hence $M_{\mcl K} \Psi$ are incrementally bounded with bounds $C_{\Psi}$ and $C_{M_{\mcl K} \Psi}=C_{\Psi} \norm{{\mcl K}}$, respectively, and also by Cauchy Schwartz inequality, we have
\begin{align*}
&\left<\Psi \bmat{ {\mbf v} \\ \D {\mbf v}} -\Psi \bmat{ G {\mbf u}\\ {\mbf u}}, M_{\mcl K} \Psi  \bmat{ {\mbf v} \\ \D {\mbf v}} -M_{\mcl K} \Psi \bmat{ G {\mbf u}\\ {\mbf u}}\right>_T\\
&\hspace{-3mm}\qquad \le \norm{\Psi \bmat{ {\mbf v} \\ \D {\mbf v}}\hspace{-1mm} - \Psi \bmat{ G {\mbf u}\\ {\mbf u}}} \norm{ M_{\mcl K} \Psi  \bmat{ {\mbf v} \\ \D {\mbf v}}\hspace{-1mm} - M_{\mcl K} \Psi \bmat{ G {\mbf u}\\ {\mbf u}}}_T \\
&\qquad  \le C_\Psi C_{M_{\mcl K} \Psi}\left\| \bmat{ {\mbf v} - G {\mbf u}\\ \D {\mbf v} - {\mbf u}}_T\right\|_{\mbf L}^2.
\end{align*}
Similarly, because $G$ is bounded with bound $C_G$,
 \begin{align*}
& \left<\Psi \bmat{ G {\mbf u} \\  {\mbf u} }, M_{\mcl K} \Psi  \bmat{{\mbf v} \\ \D {\mbf v}} - M_{\mcl K} \Psi \bmat{ G {\mbf u}\\ {\mbf u}}\right>_T \\
& \qquad \le C_\Psi C_{M_{\mcl K} \Psi}\left\|\bmat{G {\mbf u} \\ {\mbf u}}_T \right\|_{\mbf L} \left\| \bmat{ {\mbf v} - G {\mbf u}\\ \D {\mbf v} - {\mbf u}}_T\right\|_{\mbf L}
%& \qquad  \le C_\Psi C_{M_{\mcl K} \Psi}(C_G+1)\left\|{\mbf u}_T \right\|_{\mbf L} \left\| \bmat{ {\mbf v} - G {\mbf u}\\ \D {\mbf v} - {\mbf u}}_T\right\|_{\mbf L_{\mbf Z}}
\end{align*}
and
 \begin{align*}
 & \left<\Psi \bmat{{\mbf v} \\ \D {\mbf v}} - \Psi \bmat{  G {\mbf u}\\  {\mbf u}}, M_{\mcl K} \Psi  \bmat{ G {\mbf u}\\ {\mbf u}}\right>_T \\
 & \qquad \le C_\Psi C_{M_{\mcl K} \Psi}\left\|\bmat{G {\mbf u} \\ {\mbf u}}_T \right\|_{\mbf L} \left\| \bmat{ {\mbf v} - G {\mbf u}\\ \D {\mbf v} - {\mbf u}}_T\right\|_{\mbf L}.
\end{align*}
Since $\mbf v=G\mbf u+\f$ and $\mbf u=\Delta \mbf v+\e$, we conclude that
 \begin{align*}
   \varepsilon \| {\mbf u}_T\|_{\mbf L}^2 &\leq  C_\Psi C_{M_{\mcl K} \Psi}\bbbl(\left\| \bmat{ {\mbf v} - G {\mbf u}\\ \D {\mbf v} - {\mbf u}}_T\right\|_{\mbf L}^2 + \\
   &\qquad \qquad + 2 \left\|\bmat{G {\mbf u} \\ {\mbf u}}_T \right\|_{\mbf L} \left\| \bmat{ {\mbf v} - G {\mbf u}\\ \D {\mbf v} - {\mbf u}}_T\right\|_{\mbf L} \bbbr) \\
   &\hspace{-8mm}=  C_\Psi C_{M_{\mcl K} \Psi}\left(\left\| \bmat{ \f\\ \e}_T\right\|_{\mbf L}^2 +  2 \left\|\bmat{G {\mbf u} \\ {\mbf u}}_T \right\|_{\mbf L} \left\| \bmat{ \f\\ \e}_T\right\|_{\mbf L} \right) \\
&\hspace{-12	mm}\le  C_\Psi C_{M_{\mcl K} \Psi}\left(\left\| \bmat{ \f\\ \e}_T\right\|_{\mbf L}^2 +  2 (C_G+1)\norm{\mbf u_T}_{\mbf L} \left\| \bmat{ \f\\ \e}_T\right\|_{\mbf L} \right).
\end{align*}
Next,  if we complete the square by adding $ C_\Psi C_{M_{\mcl K} \Psi}(C_G+1)^2\norm{\mbf u_T}^2_{\mbf L}$ to both sides, we get
\begin{align*}
 \varepsilon \| {\mbf u}_T\|_{\mbf L}^2 & + C_\Psi C_{M_{\mcl K} \Psi} (C_G+1)^2 \norm{{\mbf u}_T}^2_{\mbf L}  \\
% &\hspace{-8mm}\le C_\Psi C_{M_{\mcl K} \Psi} \bbbl(\left\| \bmat{ \f\\ \e}_T\right\|_{\mbf L}^2 \\
% &\qquad \qquad  +  2 (C_G+1)\norm{\mbf u_T}_{\mbf L} \left\| \bmat{ \f\\ \e}_T\right\|_{\mbf L}\\
% &\qquad \qquad \qquad\qquad\qquad+  (C_G+1)^2 \norm{\mbf u_T}_{\mbf L}\bbbr)  \\
 &= C_\Psi C_{M_{\mcl K} \Psi}\left(\left\| \bmat{ \f\\ \e}_T\right\|_{\mbf L} +  (C_G+1)\norm{\mbf u_T}_{\mbf L}\right)^2.
\end{align*}
Hence
\begin{align*}
&\sqrt{\frac{\varepsilon}{C_\Psi C_{M_{\mcl K} \Psi}} +  (C_G+1)^2} \| {\mbf u}_T\|_{\mbf L}  \\
&\qquad \le   \left\| \bmat{ \f\\ \e}_T\right\|_{\mbf L} +  (C_G+1)\norm{\mbf u_T}_{\mbf L}
\end{align*}
or for $\hat \varepsilon = \frac{\varepsilon}{C_\Psi C_{M_{\mcl K} \Psi}}$ and $K = (C_G+1)$ we have
\[
\left(\sqrt{\hat \varepsilon + K^2} - K \right) \| {\mbf u}_T\|_{\mbf L}  \leq  \left\| \bmat{ \f\\ \e}_T\right\|_{\mbf L}.
\]
Then, defining $C = \left(\sqrt{\hat \varepsilon + K^2} - K \right)^{-1}>0$ we have
\[
\|\mbf u_T\|_{\mbf L} \leq C  \left\| \bmat{ \f\\ \e}_T\right\|_{\mbf L}.
\]
Finally, we define $C_{\mbf v} := C_g C + 1$. Thus we have
%since $\mbf u$ is upper-bounded by $\mbf{e}, \mbf f$ and $G$ is bounded, then there exists $C_{\mbf v} > 0$ such that
\[
\|\mbf v_T\|_{\mbf L} \leq C_G \|\mbf u_T\|_{\mbf L} + \|f\|_{\mbf L} \leq C_{\mbf v}   \left\| \bmat{ \f\\ \e}_T\right\|_{\mbf L}
\]
for all $T>0$.

 We conclude that the interconnection is stable.

\end{proof}

%Having shown that the conditions of Theorem~\ref{thm:IQC} imply stability of the feedback interconnection of $G$ and $\Delta$, we now consider how to test these conditions when the signal spaces are infinite-dimensional.

% We note that an alternative Lyapunov-based stability theorem based solely on $M_{\mcl K}$ can be found in~\cite{das2020robust}. {\color{red} Authors use quadratic constraints to show internal stability of interconnection. }

%{\color{red} $({\mcl K},\Psi)\in \mbf Q:=\{({\mcl K},\Psi)=\sum_i (\alpha_i {\mcl K}_i, \alpha_i \Psi_i\\;:\; \alpha_i >0\}$,}

\textbf{Remark:} For a given class of $\Delta$, we are typically given a set of valid $\mcl K$ and $\Psi$ (e.g. Lemma~\ref{lem:constant}). While it is easy to search over a convex set of $\mcl K$ for a given $\Psi$, it is not as easy to search over all possible valid $\Psi$ multipliers. A typical approach, therefore, is to choose a collection of IQCs $\{\mathcal{K}_i,\Psi_i\}_{i=1}^n$ and note that any conic combination of the following form is also a valid IQC:
\[
\Psi=\bmat{\Psi_1 \\ \vdots \\ \Psi_n} \quad \mcl K(\lambda) = \bmat{\lambda_1 \mcl K_1 & & \\
                                                                                                 & \ddots & \\
                                                                                                 &  & \lambda_n \mcl K_n},
\]
where $\{\lambda_i\}_{i=1}^n$ are any non-negative constants.

The set of $\mathcal{K}(\lambda)$ is convex thus allowing convex optimization methods to search for feasible values of $\lambda_i$.  Other convex parameterizations exists for certain classes of IQC multipliers \cite{VEENMAN20161}.

Our goal, then, is to find some $\mcl K(\lambda), \Psi$ for which Eqn.~\eqref{IQC:G} is satisfied. To achieve this goal, we require some way of characterizing the input-output properties of the multiplier-mapped graph $\Psi \bmat{G \mbf u\\ \mbf u}$. For this problem, we turn to state-space representations and the KYP lemma as extended to infinite-dimensional systems in the Partial Integral Equation (PIE) framework.

\section{Input-Output Analysis using PIEs and the KYP Lemma}
\label{sec:PIETOOLS}

%To apply Theorem~\ref{thm:IQC} to stability of nonlinear and uncertain systems, we suppose that we are given a set, $\mbf \D$, of possible nonlinear or uncertain subsystems and that we have  $\mcl K(\lambda) \in \mcl L( \mbf Z^{\n + \m})$ and $\Psi \in \mcl L(\mbf L^{\n + \m}_{e, [a, b]})$ such that any $\Delta \in \mbf \Delta$ satisfies the Hard IQC defined by operators $\Psi$ and $\mcl K(\lambda)$ for all $\lambda \in \Lambda$. %This implies that for any $({\mcl K},\Psi)\in  Q$  Eqn.~\eqref{IQC:D} is satisfied for all $\Delta \in \mbf \Delta$, where

%Our goal, then, is to find some $\mcl K(\lambda), \Psi$ for which Eqn.~\eqref{IQC:G} is satisfied, but it requires a way of characterizing the input-output properties of the multiplier-mapped graph $\Psi \bmat{G \mbf u\\ \mbf u}$. For this problem, we turn to state-space representations and the KYP lemma as extended to infinite-dimensional systems in the Partial Integral Equation (PIE) framework. % This problem is convex,
In this section, we propose a method of using convex optimization to test the conditions of Theorem~\ref{thm:IQC}. This is accomplished in three parts. First, we assume the nominal system and multipliers, $\Psi$, are represented as Partial Integral Equations (PIEs). Second, we generalize the KYP lemma to PIEs. Finally, we pose the conditions of Theorem~\ref{thm:IQC} as a convex optimization problem over the cone of positive Partial Integral (PI) operators.

\subsection{Partial Integral Equations}
Our method for testing the conditions of Theorem~\ref{thm:IQC} assumes there exists a state-space representation of the systems $G,\Psi:\mbf u \mapsto \mbf y$ of the form
\begin{align}
\mcl T \dot{\mbf x}(t) & = \mcl A \mbf x(t) + \mcl B \mbf u(t) \label{eqn:G}\\
	 	 \mbf y(t)   & = \mcl C \mbf x(t) + \mcl D \mbf u(t) \notag,
\end{align}
where  $\mbf x(t) \in \mbf Z^{\mbf k} $,  $\mbf y(t) \in \mbf Z^{\m}  $, $\mbf u(t) \in \mbf Z^{\n}  $  and $\mcl T, \mcl A, \mcl B, \mcl C, \mcl D$  are PI operators with appropriate dimensions.

\noindent\textbf{Remark:} Note that most linear delayed and well-posed PDE systems can be represented in this form --- See~\cite{peet2021partial} for PDE and~\cite{peet2021representation} for delayed systems.

When such a representation exists, it is referred to as a Partial Integral Equation (PIE).
We use the PIE representation~\eqref{eqn:G} as it is possible to optimize over the cone of positive PI operators using, e.g.~\cite{shivakumar2020pietools}. This allows us to test the conditions of the the following infinite-dimensional version of the KYP lemma.

\subsection{KYP lemma}

To test the conditions of Theorem~\ref{thm:IQC}, we presume that $\Delta \in \mbf \Delta$ satisfies the IQC for some $K\in \Pi_4$ and $\Psi$, where $\Psi$ admits a PIE representation with parameters $\mcl T_\Psi, \mcl A_\Psi, \mcl B_\Psi, \mcl C_\Psi \in \Psi_4$.

 Assuming for now that $\Psi=I$, the following Lemma provides conditions under which $G$ satisfies the conditions of Theorem~\ref{thm:IQC}.
%The following Lemma provides conditions under which $G$ satisfies the conditions of Theorem~\ref{thm:IQC}.

 \begin{lem}[sufficient version of KYP lemma] \label{lem:KYP}
 Suppose $G\in \mcl L(\mbf L^{\n }_{e, [a, b]} ,\mbf L^{\  \m}_{e, [a, b]} )$ is a causal bounded linear operator and there exist  $\mcl T, \mcl A, \mcl B, \mcl C, \mcl D\in \Pi_4$ such that for any $\mbf u \in \mbf L^\n_{e, [a, b]}$, $\mbf y = G \mbf u$ implies that $\mbf y$ satisfies~\eqref{eqn:G} for some $\mbf x \in \mbf L^{\mbf k}_{e, [a, b]}$. Given $\mcl K \in \Pi_4$, suppose there exists some $\varepsilon>0$ and $\mcl P \in \Pi_4$ such that $\mcl P \ge 0$ and
 \begin{align}
 \bmat{ \mcl T^*\mcl P\mcl A + \mcl A^*\mcl  P \mcl T &\mcl  P\mcl  B \\ \mcl B^*\mcl P & \varepsilon I} + \bmat{\mcl C^* \\ \mcl D^*} \mcl K \bmat{\mcl C & \mcl  D} \leq 0. \label{eqn:KYP-lemma}
\end{align}
Then for any $ \mbf u \in \mbf L^{\n}_{e, [a, b]}$ we have that
\begin{equation*}
      %& \left<P_T \Psi \bmat{ {\mbf v}\\ \D {\mbf v}}, P_T M_{\mcl K} \Psi \bmat{ {\mbf v}\\ \D {\mbf v}}\right>_{\mbf L} \geq 0 \label{IQC:D}  \\
     \hspace{-10mm} \left<P_T  G{\mbf u} , P_T M_{\mcl K}   G{\mbf u}\right>_{\mbf L} \leq -\varepsilon\|{P_T \mbf u}\|^2_{\textbf{L}},
\end{equation*}
    % \left<\Psi \bmat{ {\mbf v}\\ \D {\mbf v}}, {\mcl K} \bmat{ {\mbf v}\\ \D {\mbf v}}\right>_T \geq 0 \;\;\text{ for all } {\mbf v}\in \mbf L_{e, [a,b]}^n[0, \infty)
    % \]
    % \[
    % \left<\Psi \bmat{ G{\mbf u}\\ {\mbf u}}, {\mcl K} \bmat{ G{\mbf u}\\ {\mbf u}}\right>_T \leq -\{\mbf v}arepsilon\|{\mbf u}_T\|_{L_{H,e}} \;\;\text{ for all } {\mbf u}\in \mbf L_{e, [a,b]}^n[0, \infty)
    % \]
%\end{itemize}
for all $T>0$.

%Then for any $T > 0$, we have that
%\[
%\int_0^T \left<\mbf y(t), \mcl K \mbf y(t)\right>_{\mbf Z}dt = \int_0^T\left<(G \mbf u)(t), \mcl K (G \mbf u)(t)\right>_{\mbf Z} dt \leq 0.
%\]
 \end{lem}

\begin{proof} Define $V(\mbf x) = \ip{\mcl T \mbf x}{\mcl P \mcl T \mbf x }_{\mbf Z}$.
 Suppose that  $  \mbf u  \in \mbf L^{ \n}_{e,  [a,b]}$ and $\mbf y(t)= (G{\mbf u})(t)$ for some $\mbf x \in \mbf L^{ \mbf k}_{e,  [a,b]}$, such that Eqns.~\eqref{eqn:G} are satisfied. By inequality~\eqref{eqn:KYP-lemma}, we have that
\begin{align*}
&\bbbbl\langle\bmat{\mbf x(t) \\ \mbf u(t)}, \bbbbl(\bmat{ \mcl T^*\mcl P\mcl A + \mcl A^*\mcl  P \mcl T &\mcl  P\mcl  B \\ \mcl B^*\mcl P & \varepsilon I}  \\
&\qquad\qquad\qquad\qquad+ \bmat{\mcl C^* \\  \mcl D^*}   \mcl K \bmat{\mcl C &  \mcl  D}\bbbbr)\bmat{\mbf x(t) \\ \mbf u(t)}\bbbbr\rangle_{\mbf Z}  \\
%& = \ip{\bmat{\mbf x(t) \\ \mbf u(t)}}{ \bmat{ \mcl T^*\mcl P\mcl A + \mcl A^*\mcl  P \mcl T &\mcl  P\mcl  B \\ \mcl B^*\mcl P & \varepsilon I}\bmat{\mbf x(t) \\ \mbf u(t)}}_{\mbf Z} \\
%& \qquad\quad+\left<\bmat{\mbf x(t) \\ \mbf u(t)},\bmat{\mcl C^* \\ \mcl D^*} \mcl K \bmat{\mcl C &\mcl  D}\bmat{\mbf x(t) \\ \mbf u(t)}\right>_{\mbf Z}  \\
& = \ip{\mcl T \dot{\mbf x}(t)}{\mcl P\mcl T \mbf x(t)} +\ip{\mcl T \mbf x(t)}{\mcl P\mcl T \dot{\mbf x}(t)}\\
&\qquad \qquad+  \varepsilon \norm{\mbf u(t)}^2_{\mbf Z} +\ip{\mbf y(y) }{\mcl K \mbf y(t)}_{\mbf Z}  \\
& = \dot V(\mbf x(t)) +\varepsilon \| \mbf u(t)\|_{\mbf Z}^2 +  \left< \mbf y(t), \mcl K \mbf y(t)\right>_{\mbf Z}\le 0.
\end{align*}
Now, since $V(\mbf x(0)) = V(0) = 0$ and $V(\mbf x(T)) \geq 0$, and integrating in time, we obtain
\begin{align*}
\varepsilon\|{P_T \mbf u}\|^2_{\textbf{L}} +   \int_0^T \left< \mbf y(t), \mcl K \mbf y(t)\right>_{\mbf Z} dt&\le -V(\mbf x(T)) + V(\mbf x(0))\\
& \le 0.
\end{align*}
We conclude that
\[
\left<P_T  G{\mbf u} , P_T M_{\mcl K}   G{\mbf u}\right>_{\mbf L} = \int_0^T \hspace{-2mm}\left< \mbf y(t), \mcl K \mbf y(t)\right>_{\mbf Z} dt\le
-\varepsilon\|{P_T \mbf u}\|^2_{\textbf{L}}.
\]

%\left<P_T  G{\mbf u} , P_T M_{\mcl K}   G{\mbf u}\right>_{\mbf L} +\varepsilon\|{P_T \mbf u}\|^2_{\textbf{L}}\\
%& \int_0^T \dot V(\mbf x(t)) dt  + \varepsilon \|P_T \mbf u\|_{\mbf L}^2+ \int_0^T \left< \mbf y(t), \mcl K \mbf y(t)\right>_{\mbf Z} dt = V(\mbf x(T)) - V(\mbf x(0)) +\varepsilon \|P_T \mbf u\|_{\mbf L}^2 + \int_0^T \left< \mbf y(t), \mcl K \mbf y(t)\right>_{\mbf Z} dt,
%\end{align*}
%where $V(\mbf x(t)) = \left<\mcl T \mbf x(t), \mcl P \mcl T \mbf x(t)  \right>_{\mbf Z}$.
%
%Since $V(\mbf x(0)) = V(0) = 0$ and $V(\mbf x(T)) \geq 0$, then
%\[
%-\varepsilon\|{P_T \mbf u}\|^2_{\textbf{L}} \geq  \int_0^T \hspace{-2mm}\left< \mbf y(t), \mcl K \mbf y(t)\right>_{\mbf Z} dt  =\left<P_T  G{\mbf u} , P_T M_{\mcl K}   G{\mbf u}\right>_{\mbf L}.
%\]
\end{proof}

%\textbf{Remark:}
%Since the nominal system is stable and well-defined by assumption in the feedback interconnection, we have a well-posed PIE representation of $G$.
%Moreover, input-output properties of linear PDE and DDE dynamics are transfered to the corresponding PI systems~\cite{shivakumar2022extension}.
%It allows us to check the IQC condition~\eqref{IQC:G} for PIEs using the KYP lemma for infinite-dimensional systems.

%\textbf{Remark:} Theorem~\ref{lem:KYP} can be used to enforce the condition of Eqn.~\eqref{IQC:G} in Theorem~\ref{thm:IQC} since if $G$ has PIE representation $\{\mcl T,\mcl A, \mcl B, \mcl C,\mcl D\}$, then $\bmat{G\\I}$ has PIE representation $\left\{\mcl T,\mcl A, \mcl B, \bmat{\mcl C\\0},\bmat{\mcl D\\ I}\right\}$.

Given PI operator $\mcl K$, the conditions of Lemma~\ref{lem:KYP} may be verified using software for optimization of PI operators as in~\cite{shivakumar2020pietools}.
%
%Moreover, input-output properties of linear PDE and DDE dynamics are transfered to the corresponding PI systems~\cite{shivakumar2022extension}. It allows us to check the IQC condition~\eqref{IQC:G} for PIEs using the KYP lemma for infinite-dimensional systems.

\subsection{Augmentation of the Dynamics}
Now, we suppose that the multiplier $\Psi \in \mcl L(\mbf L^{\n + \m}_{e, [a, b]})$ also admits a PIE representation of the form
\begin{align}
\mcl T_\Psi \dot{\mbf z}(t) & = \mcl A_\Psi \mbf z(t) + \mcl B_\Psi \mbf v(t) \label{eqn:Psi}\\
	 	 \mbf w(t)   & = \mcl C_\Psi \mbf z(t) + \mcl D_\Psi \mbf v(t). \notag
\end{align}

 \begin{cor}[Augmented KYP lemma] \label{cor:KYP}
 Suppose $G\in \mcl L(\mbf L^{\n }_{e, [a, b]} ,\mbf L^{\  \m}_{e, [a, b]} )$ is a causal bounded linear operator and there exist  $\mcl T, \mcl A, \mcl B, \mcl C, \mcl D\in \Pi_4$ such that for any $\mbf u \in \mbf L^\n_{e, [a, b]}$, $\mbf y = G \mbf u$ implies that $\mbf y$ satisfies~\eqref{eqn:G} for some $\mbf x \in \mbf L^{\mbf k_G}_{e, [a, b]}$.

 Furthermore, suppose that there exist  $\mcl T_\Psi, \mcl A_\Psi, \mcl B_\Psi, \mcl C_\Psi, \mcl D_\Psi\in \Pi_4$ such that for any $\mbf v \in \mbf L^{\n+\m}_{e, [a, b]}$, $\mbf w =\Psi \mbf v$ implies that $\mbf w$ satisfies~\eqref{eqn:Psi} for some $\mbf z \in \mbf L^{\mbf k_\Psi}_{e, [a, b]}$.

Given $\mcl K \in \Pi_4$, suppose there exists some $\mcl P \in \Pi_4$ such that $\mcl P \ge 0$ and
 \begin{align}
 \bmat{ \hat{\mcl T}^*\mcl P \hat{\mcl A} +  \hat{\mcl A}^*\mcl  P \hat{\mcl T} &\mcl  P \hat{\mcl  B} \\  \hat{\mcl B}^*\mcl P & \varepsilon I} + \bmat{ \hat{\mcl C}^* \\  \hat{\mcl D}^*} \mcl K \bmat{ \hat{\mcl C} & \hat{\mcl  D}} \leq 0.  \label{eqn:KYP_PSIG}
\end{align}
where
\begin{align}
\hat{\mcl T} = \bmat{{\mcl T} & 0 \\ 0 & {\mcl T}_\psi}, \;
&\hat{\mcl A} = \bmat{\mcl A & 0 \\ {\mcl B}_\Psi\hspace{-1mm}\bmat{\mcl C\\0} & {\mcl A}_\Psi}, \; \hat{\mcl B} = \bmat{{\mcl B} \\ {\mcl B}_\Psi\hspace{-1mm}\bmat{D \\ I}}, \notag \\
&\hspace{-10mm}\hat{\mcl C} =\bmat{{\mcl D}_\Psi\hspace{-1mm}\bmat{\mcl C\\0} & {\mcl C}_\Psi}, \quad
\hat{\mcl D} = {\mcl D}_\Psi\hspace{-1mm}\bmat{\mcl D \\ I}. \label{eqn:PsiG}
\end{align}
Then for any $ \mbf u \in \mbf L^{\n}_{e, [a, b]}, $ we have that
\begin{equation*}
     \hspace{-10mm} \left<P_T \Psi \bmat{ G{\mbf u}\\ {\mbf u}}, P_T M_{\mcl K}  \Psi\bmat{ G{\mbf u}\\ {\mbf u}}\right>_{\mbf L} \leq -\varepsilon\|{P_T \mbf u}\|_{\textbf{L}},
\end{equation*}
for all $T>0$.
 \end{cor}
\begin{proof} The proof follows immediately from Lemma~\ref{lem:KYP} since if $G$ has PIE representation $\{\mcl T,\mcl A, \mcl B, \mcl C,\mcl D\}$ and $\Psi$ has PIE representation $\{\mcl T_\Psi, \mcl A_\Psi, \mcl B_\Psi, \mcl C_\Psi, \mcl D_\Psi\}$, then $\Psi \bmat{G\\I}$ has PIE representation $\{\hat{\mcl T},\hat{\mcl A}, \hat{\mcl B}, \hat{\mcl C},\hat{\mcl D}\}$ -- i.e. if $\mbf w=\Psi \bmat{G\\I}\mbf u$, then for some $\mbf x,\mbf z$,
\begin{align*}
\bmat{{\mcl T} &\hspace{-1mm} 0 \\ 0 &\hspace{-1mm} {\mcl T}_\psi} \bmat{\dot{\mbf x}(t)\\\dot{\mbf z}(t)} & = \bmat{\mcl A & \hspace{-1mm} 0 \\ {\mcl B}_\Psi \hspace{-1mm}\bmat{\mcl C\\0} &\hspace{-1mm} {\mcl A}_\Psi}\hspace{-1mm} \bmat{ {\mbf x}(t)\\ {\mbf z}(t)} + \bmat{{\mcl B} \\ {\mcl B}_\Psi\hspace{-1mm} \bmat{D \\ I}} \mbf u(t) \label{eqn:Psi}\\
	 	 \mbf w(t)   & =\bmat{{\mcl D}_\Psi\hspace{-1mm} \bmat{\mcl C\\0} & {\mcl C}_\Psi} \hspace{-1mm}\bmat{ {\mbf x}(t)\\ {\mbf z}(t)} + {\mcl D}_\Psi \hspace{-1mm}\bmat{\mcl D \\ I} \mbf u(t). \notag
\end{align*}

%Since for all $\mbf u \in \mbf L^{\n}_{e, [a, b]}$ and $\mbf v \in \mbf L_{e, [a, b]}^{\n + \m}$ we have $\mbf y = G\mbf u$ satisfies equation~\eqref{eqn:G} and $\mbf w = \Psi\mbf v$ satisfies Equation~\eqref{eqn:Psi} for some $\mbf x \in \mbf L^{\mbf k_G}_{e, [a, b]}$ and $\mbf z \in \mbf L^{\mbf k_\Psi}_{e, [a, b]}$. , then $\mbf w = \Psi \bmat{G\\I} \mbf u$ satisfies
%\begin{align*}
%\hat{\mcl T} \bmat{\dot{\mbf x}(t)\\\dot{\mbf z}(t)} & = \hat{\mcl A} \bmat{ {\mbf x}(t)\\ {\mbf z}(t)} + \hat{\mcl B} \mbf u(t) \label{eqn:Psi}\\
%	 	 \mbf w(t)   & = \hat{\mcl C} \bmat{ {\mbf x}(t)\\ {\mbf z}(t)} + \hat{\mcl D} \mbf u(t). \notag
%\end{align*}
%Therefore, we can apply Lemma~\ref{lem:KYP} to show that
% \begin{equation*}
%      \left<P_T \Psi \bmat{ G{\mbf u}\\ {\mbf u}}, P_T M_{\mcl K}  \Psi\bmat{ G{\mbf u}\\ {\mbf u}}\right>_{\mbf L} \leq -\varepsilon\|{P_T \mbf u}\|_{\textbf{L}},
%\end{equation*}
%for all $T>0$.
\end{proof}

%\textbf{Remark 1.} The KYP lemma allows us to construct the quadratic storage function $V:\mbf Z^m \rightarrow \R$ for IQC theorem~\ref{thm:IQC}. Specifically,  if the IQC is defined by operators $\Psi$, $\mcl K$, where $\Psi$ is a some PIE system, then $\Psi \bmat{G\\I}$ is also PIE system with known state space PI realization. Next,  Suppose that  $\mbf z$ is a state-space variable of the operator $\Psi \bmat{G\\I}$, then $V = \ip{\mcl T \mbf z}{\mcl P\mcl T \mbf z}$ for some PI operator $\mcl P$.
%Note that if $\mbf \D$ is large enough and It can not be analyzed by some $\Psi$, then we can use storage functions $\mcl P_i$ and $\Psi_i$ to characterize the stability region for $\mbf \D_i \subset \mbf \D$.
%
%\textbf{Remark 2.} Suppose $\Lambda\subset \R$ is a convex set,  $ \mcl K:\Lambda \rightarrow \mcl L(\mbf Z^{\n + \m})$ is a linear function of $\lambda$, $\Psi \in \mcl L(\mbf L^{\n + \m}_{e, [a, b]})$ and $\D$ satisfies the hard IQC defined by $\mcl K(\lambda)$, $\Psi$ for all $\lambda\in \Lambda$. Then since $\mcl P$ can be parametrized by polynomials and $ \mcl K$ is a linear function $\lambda$, then we can formulate a convex problem to solve inequality~\eqref{eqn:KYP-lemma}.

\subsection{Testing the Conditions of Theorem~\ref{thm:IQC}}
We now suppose the existence of a PIE representation of $G$ and $\Psi$ and propose a convex optimization problem whose feasibility verifies the conditions of Thm.~\ref{thm:IQC}.
\begin{thm}\label{thm:KYP_and_IQC}
Suppose  $G\in \mcl L(\mbf L^{\n }_{e, [a, b]} ,\mbf L^{\  \m}_{e, [a, b]} )$ is a causal bounded linear operator and for $\mbf u \in \mbf L^\n_{e, [a, b]}$,  $\mbf y = G \mbf u$ implies Eqns.~\eqref{eqn:G} are satisfied for $\{\mcl T,\mcl A, \mcl B, \mcl C,\mcl D\}$ and some $\mbf x \in \mbf L^{\mbf k_G}_{e, [a, b]}$.

Further suppose that for any  $\D \in \mbf \D$, the interconnection defined by $[G, \D]$ is well-posed and $\Delta$ satisfies the Hard IQC defined by $\mcl K\in \Pi_4$ and $\Psi$ where $\mbf w = \Psi \mbf v$ implies Eqns.~\eqref{eqn:Psi} are satisfied for $\{\mcl T_\Psi, \mcl A_\Psi, \mcl B_\Psi, \mcl C_\Psi, \mcl D_\Psi\}$ and some $\mbf z \in \mbf L_{e, [a, b]}^{\mbf k_\Psi}$.

Then we have the following.
\begin{enumerate}
\item If there exists $\mcl P \in \Pi_4$ such that $\mcl P \ge 0$ and Inequality~\eqref{eqn:KYP_PSIG} holds for $\{\hat{\mcl T},\hat{\mcl A},\hat{\mcl B},\hat{\mcl C},\hat{\mcl D}\}$ as defined in~\eqref{eqn:PsiG}, we have that the feedback interconnection defined by $[G, \D]$ is stable for all $\D \in \mbf \D$.
% \begin{align*}
% \bmat{ \mcl T^*\mcl P\mcl A + \mcl A^*\mcl  P \hat{\mcl T} &\mcl  P\mcl  B \\ \mcl B^*\mcl P & \varepsilon I} + \bmat{\mcl C^* \\ \mcl D^*} \mcl K \bmat{\mcl C & \mcl  D} \leq 0,
%\end{align*}
\end{enumerate}
%where $\hat{\mcl T}$,$\hat{\mcl A}$,$\hat{\mcl B}$,$\hat{\mcl C}$ and $\hat{\mcl D}$ as defined in~\eqref{eqn:PsiG}.

\end{thm}
\begin{proof}
Suppose there exists $\mcl P \in \Pi_4$ such that $\mcl P \ge 0$ and Inequality~\eqref{eqn:KYP_PSIG} is satisfied. As per Corollary~\ref{cor:KYP}, we have that Inequality~\eqref{IQC:G} holds for all $\mbf u \in \mbf L^\n_{e, [a, b]}$ -- i.e.
 \begin{equation*}
     \left<P_T \Psi \bmat{ G{\mbf u}\\ {\mbf u}}, P_T M_{\mcl K}  \Psi\bmat{ G{\mbf u}\\ {\mbf u}}\right>_{\mbf L} \leq -\varepsilon\|{P_T \mbf u}\|_{\textbf{L}}.
\end{equation*}
Since any $\D \in \mbf \D$ satisfies the Hard IQC~\eqref{IQC:G}, Theorem~\ref{thm:IQC} implies that feedback system defined by $[G, \D]$ is stable for any $\D \in \mbf \D$.
\end{proof}

\section{Types of IQC}
\label{sec:typesIQC}
In Section~\ref{sec:PIETOOLS}, we have assumed that the causal uncertain or nonlinear subsystem $\D:\mbf L^{\m}_{e, [a, b]} \rightarrow \mbf L^\n_{e, [a, b]}$ is known a priori to satisfy a hard IQC defined by some $\mcl K$ and $\Psi$. As is typical in the finite-dimensional case, the set of $\mcl K$ and $\Psi$ for which the hard IQC hold are determined by the input-output properties of $\Delta$. In this section, we review the infinite-dimensional equivalent of several well-studied classes of uncertainty/nonlinearity and provide corresponding infinite-dimensional extensions of the relevant finite-dimensional IQCs.

% Although $\D$ is unknown operator, the set, $\mbf \D$,  of all possible $\D$ is usually given by the problem formulation.
%%Now, the state space realization is
%%%We assume that $\mbf 0 \in \mbf \D$, so It allows us to apply the IQC framework.
%%%The main goal of this paper is to study the stability of the closed loop system $[G, \D]$ described by Eqn.~\eqref{eqn:GD}
%%\begin{align}
%%\mcl T \dot{\mbf x(t)} & = \mcl A \mbf x(t) + \mcl B \mbf u(t) \notag\\
%%	 	 \mbf y(t)   & = \mcl C \mbf x(t) + \mcl D \mbf u(t) \label{eqn:GD} \\
%%	 	 \mbf u(t)   & = (\D \mbf y)(t). \notag
%%\end{align}
%Unfortunately, not all uncertainties can be represented as PI-operators, so we present several types of IQC to describe $\mbf \D$ as close as possible.
%%Then the IQC conditions for $G$ system have to be verified via sufficient condition of KYP lemma.

\subsection{Real Constant Multiplication}
\begin{lem} \label{lem:constant}
Suppose that $(\Delta\mbf v)(t)=\delta \mbf v(t)$ for some $\delta \in \R$ such that $|\delta| \le 1$. Then for any $\mathcal{P}, \mathcal{R} \in \Pi_4$ such that $\mathcal{R}^* = - \mathcal{R}$, $\mathcal{P}^* =  \mathcal{P}\ge 0$ and for any causal bounded linear $H\in \mcl L (\mbf L^{\n + \m}_{e, [a, b]})$ we have that $\Delta$ satisfies the Hard IQC defined by
\begin{equation}
 \mcl K = \bmat{ \mathcal{P} & \mathcal{R} \\ \mathcal{R}^* & -\mathcal{P}}\quad \text{ and } \quad  \Psi = \bmat{  H & 0 \\ 0 &   H}.\label{eqn:constant}
 \end{equation}
\end{lem}
\vspace{1mm}
\begin{proof}
%Let $\Delta \mbf v(t) = \delta \mbf v(t)$ for all $\mbf v(t) \in \mbf Z$ and $|\delta| < 1$, then
For any $\mathcal{P}, \mathcal{R} \in \Pi_4$ such that $\mathcal{R}^* = - \mathcal{R}$, $\mathcal{P}^* =  \mathcal{P}$ and for any causal bounded linear $H \in \mcl L(\mbf L^{\n + \m}_{e, [a, b]})$ we have that
\begin{align*}
     &\left<P_T \Psi \bmat{\mbf v\\ \D \mbf v}, P_T  M_{\mcl K}\Psi \bmat{\mbf v\\ \D\mbf v}\right>  \\
     &\qquad  =  \int_0^T \left<   \bmat{  H\mbf v \\ H \D\mbf v}(t) , M_{\mcl K}\bmat{  H\mbf v \\ H\D\mbf v}(t)\right>_{\mbf Z} dt \\
     &\qquad  =  \int_0^T (1 - \delta^2) \left< (H\mbf v)(t) , M_{\mcl P} (H \mbf v)(t)\right>_{\mbf Z}dt\\& \qquad \ge 0.
% -  \delta^2 \left< (H \mbf v)(t) , M_{\mcl P}  (H \mbf v)(t)\right>_{\mbf Z} dt
\end{align*}
Therefore  $\Delta$ satisfies the IQC defined by
\[
 \mcl K = \bmat{ \mathcal{P} & \mathcal{R} \\ \mathcal{R}^* & -\mathcal{P}}\quad \text{ and } \quad \Psi = \bmat{  H & 0 \\ 0 &   H}.
 \]
\end{proof}
%\textbf{Remark:} If $H$ has PIE representation $\left\{\mcl T,\mcl A, \mcl B, \mcl C,\mcl D \right\}$, then $\bmat{H &0\\0&H}$ has PIE representation
%\[\left\{\bmat{\mcl T&0\\0&\mcl T},\bmat{\mcl A&0\\0&\mcl A}, \bmat{\mcl B&0\\0 &\mcl B}, \bmat{\mcl C&0\\0& \mcl C},\bmat{\mcl D&0\\0 &\mcl D}\right\}\]

\begin{cor}\label{cor:timeconstant} Suppose that $ (\Delta\mbf v)(t) = \delta(t) \mbf v(t)$ for some $\delta:\R\rightarrow \R$ such that $\sup_{t>0}|\delta(t)| \leq 1$.

Then for any $\mathcal{P}, \mathcal{R} \in \Pi_4$ such that $\mathcal{R}^* = - \mathcal{R}$ and $\mathcal{P}^* =  \mathcal{P} \geq 0$ we have that $\D$ satisfies the Hard IQC defined by
\[
 \mcl K = \bmat{ \mathcal{P} & \mathcal{R} \\ \mathcal{R}^* & -\mathcal{P}} \quad \text{ and } \quad \Psi = I,
 \]
\end{cor}
\begin{proof}
 The proof is similar to that of Lemma~\ref{lem:constant}
\end{proof}
%
%Similarly, if $\Delta \mbf v(t) = \delta(t) \mbf v(t)$ such that $\sup_{t>0} |\delta(t)| < 1$, then $\D$ satisfies the IQC defined by
%\[
% \mcl K = \bmat{ \mathcal{P} & \mathcal{R} \\ \mathcal{R}^* & -\mathcal{P}}\qquad \text{ and } \qquad \Psi = I,
% \]
%%%%%%%%%%%%%%%%%%%%%%%%%%%%%%%%%
\subsection{Polytopic uncertainty}

\begin{lem}\label{lem:polytope}
Let $\mbf \D := \{\sum_i \mu_i \D_i\;:\;\sum_i \mu_i=1 \}$ where $\D_i \in \Pi_4$. Then $\D$ satisfies the Hard IQC defined by
\[
\mcl K =  \bmat{\mcl P & \mcl R \\ \mcl R^* & \mcl Q} \quad \text{and} \quad \Psi = I
\]
where  $\mcl P, \mcl R, \mcl Q \in \Pi_4$ are such that $\mcl Q < 0$ and
\begin{align}
P + \D_i^* R^* + R \D_i + \D_i^* Q \D_i \geq 0 \quad \text{for all } i. \label{eqn:polytope}
\end{align}
\end{lem}
\begin{proof}
Suppose  $\mcl P, \mcl R, \mcl Q \in \Pi_4$ are such that $\mcl Q < 0$ and Inequality~\eqref{eqn:polytope} is satisfied. Then for any $\D \in \mbf \D$ and $\mbf v \in \mbf L_{e, [a, b]}^\m$ we have that
\begin{align*}
& \ip{\bmat{\mbf v(t) \\ (\D \mbf v)(t)}}{  \bmat{\mcl P & \mcl R \\ \mcl R^* & \mcl Q}\bmat{\mbf v(t) \\ (\D \mbf v)(t)}}_{\mbf Z}\\
& = \ip{\mbf v(t)}{ P + \D^* R^* + R \D + \D^* Q \D\mbf v(t)}_{\mbf Z}.
\end{align*}
Thus, by the convexity of the set $\mbf \D$ and $\mcl Q < 0$ we have that any $\D \in \mbf \D$ satisfies the IQC defined by
\[
\mcl K =  \bmat{\mcl P & \mcl R \\ \mcl R^* & \mcl Q} \quad \text{and} \quad \Psi = I.
\]
\end{proof}
%\begin{proof}
%Since $\|\D \mbf u\| \leq \|\mbf u\|$ for all $\mbf u \in \mbf L_{e, [a, b]}$
%\begin{align*}
%     &\left<P_T \Psi \bmat{ v\\ \D v}, P_T  {M_{\mcl K}} \Psi \bmat{ v\\ \D v}\right>_{\mbf L} \\
%    %&\qquad =  \left< \bmat{ {H} & 0 \\ 0 &  {H}}\bmat{ v\\ \D u}, \bmat{  {H} & 0 \\ 0 &  {H}}\bmat{ \mathcal{P} & 0 \\ 0 & -\mathcal{P}} \bmat{ v\\ \D v}\right>_T \\
%     &\qquad = \int_0^T  \left< \bmat{ (H\mbf v)(t)\\ (H\D \mbf v)(t)}, \bmat{ (H\mbf v)(t) \\ - (H\D \mbf v)(t)}\right>_{\mbf Z} .
%\end{align*}
%Then since $\|H\D \mbf v\| \leq \|\D \mbf v \| \leq \|v\|$
%\begin{align*}\\
%     & \qquad =    \int_0^T \left<  {H} v(t) ,  {\mcl P} (H v)(t)\right>_{\mbf Z} \\
%     & \qquad\qquad -  \int_0^T\left<  ({H} \Delta v)(t) , \mathcal{P} (H \Delta v)(t)\right>_{\mbf Z} \geq 0
%\end{align*}
%
%\end{proof}
% where $\mathcal{P} \geq 0$ is a PI operator, $H$ is a some causal bounded linear system.
%\begin{align*}
%     &\left<\Psi \bmat{ v\\ \D v}, {M_{\mcl K}} \Psi \bmat{ v\\ \D v}\right>_T \\
%    %&\qquad =  \left< \bmat{ {H} & 0 \\ 0 &  {H}}\bmat{ v\\ \D u}, \bmat{  {H} & 0 \\ 0 &  {H}}\bmat{ \mathcal{P} & 0 \\ 0 & -\mathcal{P}} \bmat{ v\\ \D v}\right>_T \\
%     &\qquad = \int_0^T  \left< \bmat{ (Hv)(t)\\ (H\D v)(t)}, \bmat{ \mcl P (Hv)(t) \\ -\mcl P (H\D v)(t)}\right>_{\mbf Z} \\
%     & \qquad =    \int_0^T \left<  {H} v(t) ,  {\mcl P} (H v)(t)\right>_{\mbf Z} \\
%     & \qquad\qquad -  \int_0^T\left<  ({H} \Delta v)(t) , \mathcal{P} (H \Delta v)(t)\right>_{\mbf Z} \geq 0
%\end{align*}
%%%%%%%%%%%%%%%%%%%%%%%%%%%%%%%
\subsection{Sector-bounded uncertainty}
\begin{lem}\label{lem:sector} Suppose that for $\mbf v \in \mbf L_{e, [a, b]}$, $(\D \mbf v)(s, t) = \phi(\mbf v(s, t))$ for all $s \in [a, b]$ and $t > 0$, where $\phi$ satisfies
\begin{equation}
 \alpha v^2 \leq v\phi(v) \leq \beta v^2\quad \text{ for all } v\in \R.
\end{equation}
Then $\D$ satisfies Hard IQC defined by $\Psi = I$ and
\begin{align}
\label{eqn:sector_IQC}
{\mcl K} &= \bmat{\beta I & - I \\ -\alpha I & I}^* \bmat{0 & \mathcal{R} \\ \mathcal{R} & 0}\bmat{\beta I & - I \\ -\alpha I & I} \\
    &\qquad = \bmat{-\beta^* \mcl R \alpha-\alpha^* \mcl R \beta & \beta^*\mcl R + \alpha^* \mcl R \\ \mcl R \beta + \mcl R \alpha & -2\mcl R} \notag,
\end{align}
for any $\mcl R\in \Pi_4$ where
\begin{align*}
(\mathcal{R}\mbf x)(s, t) & :=   R_0(s) \mbf x(s, t)
\end{align*}
 for some  $R_0(s) \geq 0$.
\end{lem}
\begin{proof}
Suppose $\mcl R \in \Pi_4$ is such that $(\mathcal{R}\mbf x)(s, t) = R_0(s) \mbf x(s, t) + \int_a^s R_1(s, \theta) \mbf x(\theta, t) d\theta + \int_s^b R_2(s, \theta) \mbf x(\theta, t) d\theta$ with $R_0(s), R_1(s, \theta) , R_2(s, \theta) \geq 0$. Then for all $s \in [a, b]$ and $t > 0$ we have that
\begin{align*}
 (\beta \mbf v(s, t) - \phi(\mbf v(s, t))) (\phi(\mbf v(s, t)) - \alpha \mbf v(s, t)) &\geq 0,
 \end{align*}
and hence we have that
\[
(\beta \mbf v(s, t) - \phi(\mbf v(s, t)))R_0(s) (\phi(\mbf v(s, t)) - \alpha \mbf v(s, t))  \geq 0.
\]
Therefore $\D$ satisfies the IQC defined by $\Psi = I$ and $\mcl K$.
\end{proof}
%\begin{align*}
%{\mcl K} &= \bmat{\beta I & - I \\ -\alpha I & I}^* \bmat{0 & \mathcal{R} \\ \mathcal{R} & 0}\bmat{\beta I & - I \\ -\alpha I & I} \\
%    &\qquad = \bmat{-\beta^* \mcl R \alpha-\alpha^* \mcl R \beta & \beta^*\mcl R + \alpha^* \mcl R \\ \mcl R \beta + \mcl R \alpha & -2\mcl R} \\
%\end{align*}
%Note that these class of IQC is known as Zames-Falb multipliers~\cite{zames1968stability}.

%\subsection{Searching over Multiple $\Psi$}
%
%
%In an alternative
%For a given $\mbf \Delta$, if  $\mbf \D$ satisfies the IQC defined by $(K,\Psi)$ for all $K \in \mbf K_i$ and $\Psi \in \mbf \Psi_{\mbf K_i}$, then $\mbf \Delta$ satisfies the Hard IQC defined by
%
%\begin{align*}
%K= \bmat{ K_1 &\hspace{-3mm} &\hspace{-3mm} \\ &\hspace{-3mm} \ddots &\hspace{-3mm} \\ &\hspace{-3mm} &\hspace{-3mm}  K_n }, \quad \Psi=\bmat{\Psi_1 \\ \vdots \\ \Psi_n}
%%\bbbbl\{( \mcl K , \Psi)=& \left( \bmat{\alpha_1 {\mcl K}_1(\lambda_1) &\hspace{-3mm} &\hspace{-3mm} \\ &\hspace{-3mm} \ddots &\hspace{-3mm} \\ &\hspace{-3mm} &\hspace{-3mm} \alpha_n {\mcl K}_n(\lambda_n) }, \bmat{\Psi_1 \\ \vdots \\ \Psi_n}\right):
%%          \begin{cases} \lambda_i \in \Lambda \\ \alpha_i > 0\end{cases} \hspace{-3mm}\bbbbr \}.
%\end{align*}
%for any $K_i \in \mbf K_i$ and $\Psi_i \in \mbf \Psi_{\mbf K_i}$. In this way, if the uncertainty set satisfies several IQCs, these IQCs can be combined when testing feasibility of the IQC constraint on $G$.

\section{Numerical examples}
%In this section we provide some IQCs and numerical examples.
\label{sec:examples}

In the following examples, we consider the problem of robust input-output stability of several systems by separating the system into nominal and uncertain subsystems and then testing the conditions of Theorem~\ref{thm:IQC} using Lemma~\ref{lem:KYP} and the software package PIETOOLS. Unless otherwise stated, the conversion of the nominal system to a PIE is performed using the conversion utilities in PIETOOLS. For simplicity, we do not include the external disturbances in the original model, although the effect of these disturbances can be inferred by the definition of the interconnection.

\subsubsection*{Example 1}
We begin with a system modeled by a reaction-diffusion equation. We would like to find the largest $\lambda_{\max}$ such that
\[
\mbf x_t(s, t) = \lambda \mbf x(s,t)+ \mbf x_{ss}(s, t)
\]
is stable for all $\lambda \in [0,\lambda_{\max}]$. For this problem, we split the dynamics into nominal and uncertain subsystems, defining the nominal $G$ by
\begin{align*}
\mbf x_t(s, t) &= \frac{\lambda_{\max}}{2}\mbf x(s,t)+\mbf x_{ss}(s, t) + \mbf u(s, t) \\
(G\mbf u)(t,s)&=\mbf x(t,s),
\end{align*}
where $s \in [0,1]$ and boundary conditions are $\mbf x(0, t) = \mbf x(1, t) = 0$. We consider the uncertain subsystem as $(\D \mbf v)(s, t) := \lambda \mbf v(s, t)$ where $\lambda \in [-\frac{\lambda_{max}}{2}, \frac{\lambda_{max}}{2}]$.

By Lemma~\ref{lem:constant}, $\Delta$ satisfies the hard IQC defined as in Eqn.~\eqref{eqn:constant} for any suitable $\mcl P,\mcl R$ and $\Psi$. For this test, we choose $\Psi$ to be defined as $\Psi := \bmat{H &0\\  0 & H }$,  where $H$ is defined as $(H \mbf y)(s,t):=\bmat{\mbf z(s, t)\\ \mbf y(s, t)}$  where
\[
\mbf z_t(s, t) = \mbf z_{ss}(s, t) + 0.5 \pi^2 \mbf z(s, t) + \mbf y(s,t) \qquad
\]
By testing the conditions of Theorem~\ref{thm:KYP_and_IQC} using PIETOOLS, it can be shown that the conditions of Eqn.~\eqref{eqn:KYP_PSIG} are feasible for $\lambda_{\max}=.99{\pi}^2$, implying that the diffusion equation is stable for any $\lambda \in [0, .99{\pi}^2]$. Note that non-robust approaches to stability analysis~\cite{valmorbida2014semi} confirm the stability interval as approximately $\lambda \in [0, \pi^2]$.

\subsection*{Example 2}
Consider the time-delay system
\begin{equation}
\dot x(t) = A x(t) +  A_d x(t- \tau) \label{eqn:time_delay},
\end{equation}
where $\tau>0$ is an uncertain delay parameter. Given $\tau_0 > 0$, the goal is to maximize $\lambda$ such that for $\tau_{\max}^{-1}=\tau_0^{-1}-\lambda$ and $\tau_{\min}^{-1}=\tau_0^{-1}+\lambda$, system~\eqref{eqn:time_delay} is stable for all $\tau \in [\tau_{min}, \tau_{max}]$ where $\tau_{\min}>0$.

For this problem, we use the nominal DDE system $G$ defined using a PIE as
\begin{align*}
\mcl T \dot{\mbf x}(t) &= \mcl A \mbf x(t) + \mcl B \mbf u(t) \\
(G \mbf u)(t) & = \mbf x(t),
\end{align*}
where $\mbf x(t) \in \mbf Z$ and
\begin{align*}
%\label{eqn:PIE_delay}
\mcl T  &= \mcl P\bmat{I & 0 \\ I & \{0, 0, I\}}, \\
\mcl A  &= \mcl P\bmat{A+   A_d & -  A_d    \\
0 & \{\tau_0^{-1}, 0 , 0\}}, \notag  \\
\mcl B  &= \mcl P\bmat{0 & 0 \\ 0 & \{\lambda , 0 , 0\}}.\notag
%\D \mbf x & =\lambda \mbf x \quad \text{with} \quad |\lambda| < 1.\notag
\end{align*}
The uncertain system $\D$ is defined as $(\D \mbf v)(t)  =\lambda \mbf v(t)$ where $|\lambda|<1$.

Note that $\D$ satisfies the Hard IQC defined in Lemma~\ref{lem:constant} as in Eqn.~\eqref{eqn:constant} for any suitable $\mcl P,\mcl R$ and $\Psi := \bmat{H &0\\  0 & H }$.

We now consider the system defined in~\cite{wu2021robust}, where
\begin{equation} \label{eqn:PIE_delay}
A = \bmat{
         0     &    1 \\
         -2     &    1
}, \qquad
A_1 = \bmat{
     0    & 0  \\
     1    & 0  \\
}
\end{equation}
In details, we use $(H \mbf y)(s,t):=\bmat{G \mbf y(s, t)\\ \mbf y(s, t)}$ and $\tau_0 = 0.189$
% where
%\[
%\mcl T \dot{\mbf z}(t) = \mcl A \mbf z(t),
%\]
%where $\{\mcl T, \mcl A\}$ is also the PI representation of the DDE system defined in Eq.~\eqref{eqn:PIE_delay}.

Using Theorem~\ref{thm:KYP_and_IQC} with the multiplier $\Psi$, as defined above, we find a robust stability region of $\tau \in [0.1008, 1.66]$ using a single storage function. Note that if we do not include the multiplier $\Psi$ (the quadratic stability case), we obtain the smaller interval $\tau \in [0.11, 0.63]$. For comparison, this problem has a known an stability range of $\tau \in [0.1, 1.717]$, that was shown in~\cite{gu2003stability}.
%. Note that using several intervals, one can achieve the delay interval exactly matching the analytical results $\tau \in [0.1, 1.717]$.

\subsubsection*{Example 3}
 The next example is a modification of a PDE in studied in~\cite{gahlawat2016convex}.
\begin{equation}\label{eqn:example3}
\mbf x_t(s, t)  = a(s)\mbf x_{ss}(s, t) + b(s)\mbf x_s(s, t) + c(s) \mbf x(s, t) + \lambda \mbf x_s (s, t),
\end{equation}
where $a(s) =s^3 - s^2 + 2 $, $b(s) =3s^2 - 2s $, $c(s) = -0.5s^3 + 1.3s^2- 1.5s + 3.03 $ and $\mbf x(0, t) =\mbf x_s(1, t) = 0$. We would like to find the maximal $\lambda_{max}$ such that the system~\eqref{eqn:example3} is stable for all $\lambda \in [-\lambda_{max}, \lambda_{max}]$.

For this task, we consider the feedback interconnection defined by $G$
 \begin{align*}
\mbf x_t(s, t)  = &a(s)\mbf x_{ss}(s, t)\hspace{-0.5mm} + b(s)\mbf x_s(s, t)\hspace{-0.5mm} + c(s) \mbf x(s, t)\hspace{-0.5mm} + \mbf u(s, t)\\
(G\mbf u)(s, t)  &= \mbf x_s(s, t).
\end{align*}

And the uncertainty is defined as $(\D \mbf v)(s, t) = \lambda \mbf v(s, t)$. Thus, $\D$ satisfies the Hard IQC defined in Lemma~\ref{lem:constant}, where we used $(H \mbf y)(s, t)  = \bmat{\mbf z(s, t) \\ \mbf y(s ,t)}$ where
\begin{align*}
\mbf z_t(s, t) &=\mbf z_{ss}(s, t) + 4.9\mbf z(s, t)+ \mbf y(s, t).
\end{align*}
Using Theorem~\ref{thm:KYP_and_IQC} and this multiplier, we may show the stability region for any $\lambda \in [-2.8, 2.8]$.
%\[
%\|\D \mbf u\| \leq \lambda\|\mbf u\| \qquad \text{for all } \mbf u \in \mbf L_{e, [a, b]}
%\]
%Next, we conclude that the interconnection is stable if $\lambda \in [-2.32, 2.32]$ using IQC theorem. In comparison, in~\cite{gahlawat2016convex} authors have demonstrated stability for $\lambda \in [-2.31, 2.31] $ and in~\cite{peet2021partial} It has been shown the stability interval for $\lambda \in [-2.33, 2.33]$

\subsubsection*{Example 4}
%Consider a nonlinear reaction-diffusion system, Zeldovich–Frank-Kamenetskii equation
%\begin{align*}
%\mbf x_t(s, t) & = \mbf  x_{ss}(s, t)  + \omega(\mbf x) \\
%\omega(\mbf x) & = \frac{\beta^2}{2}\theta (1 - \theta) e^{-\beta(1 - \theta).
%\end{align*}
The nonlinear example is adapted version~\cite{das2020robust} of the nonlinear reaction-diffusion PDE in Examples 1 and 3.
\begin{equation}\label{eqn:example4}
\mbf x_t(s, t) = \mbf x_{ss}(s, t) + \lambda \mbf x(s, t) + \phi(\mbf x(s, t)),
\end{equation}
where $\mbf x(0, t)=  \mbf x(1, t) = 0$ and the nonlinear feedback part is defined by the sector bounded function $\phi:\R\rightarrow\R$ where $- u^2\leq \phi(u)u \leq  u^2$ for $u\in \R$. The goal, then, is to find the largest $\lambda$ such that the system~\eqref{eqn:example4} is stable for any $\phi$ which satisfies the given sector bound.

First, we represent this system as the interconnection defined by $[G,\D]$, where $G$ is
\begin{align*}
\mbf x_t(s, t) &=\mbf  x_{ss}(s, t) + \lambda \mbf x(s, t) + \mbf u(s, t) \\
(G\mbf u)(s, t) & = \mbf x(s, t).
\end{align*}
Second, we define the uncertain system as $(\D \mbf v)(s ,t) = \phi(\mbf v(s, t))$.

By Lemma~\ref{lem:sector}, we have that $\D$ satisfies the Hard IQC defined by $\Psi = I$ and $\mcl K$ as in Eq.~\ref{eqn:sector_IQC}. Using Theorem~\ref{thm:IQC}, we are able to prove stability when $\lambda \leq 1.7$ -- mirroring the results in~\cite{das2020robust}.

\section{Conclusion}
In this paper, we proposed a framework for using convex optimization to study the interconnection of infinite-dimensional subsystems.
First, we extended the IQC framework to infinite-dimensional systems, signals, interconnections, and multipliers, and generalized an IQC stability theorem to such interconnections. Second, we assumed both the nominal subsystem and multiplier were represented as PIEs and extended the KYP Lemma to such systems, proposing convex tests for conditions of the IQC theorem to be satisfied. Third, we examined several classes of nonlinearity and uncertainty with infinite-dimensional inputs and outputs and showed that they satisfy a generalized version of the hard IQC constraints typically used for finite-dimensional systems. Finally, we applied the results to several example problems and showed that the proposed approach offers an improvement over alternatives such as quadratic stability.

\section*{Funding}
This work was supported by the National Science Foundation under grants No. 1931270 and 1935453.

%
% Linear PDE and DDE systems are well-parametrized by PI operators, but uncertainties, nonlinearities or time-varying parameters can cause problems to analyzes. We have propose the IQC approach, that have been formulated for infinite-dimensional input-output systems. While there are some results of using  quadratic stability, our approach covers and extends and these methods. Next, We have presented the sufficient version of KYP lemma as a way to verify the theorem conditions using PIETOOLS toolbox. Finally, we have performed numerical tests with PDEs and DDEs. These results show the improved robust stability than quadratic stability or just PI multipliers.

\addtolength{\textheight}{-12cm}   % This command serves to balance the column lengths
                                  % on the last page of the document manually. It shortens
                                  % the textheight of the last page by a suitable amount.
                                  % This command does not take   effect until the next page
                                  % so it should come on the page before the last. Make
                                  % sure that you do not shorten the textheight too much.

%%%%%%%%%%%%%%%%%%%%%%%%%%%%%%%%%%%%%%%%%%%%%%%%%%%%%%%%%%%%%%%%%%%%%%%%%%%%%%%%

%%%%%%%%%%%%%%%%%%%%%%%%%%%%%%%%%%%%%%%%%%%%%%%%%%%%%%%%%%%%%%%%%%%%%%%%%%%%%%%%

\bibliographystyle{IEEEtran}

\bibliography{Talitckii_ACC2023_v1}

\end{document}